\documentclass[11pt]{amsart}
\usepackage{geometry}                
\usepackage{graphicx,
	amssymb,
	amsmath,
	amsthm,
	bm,
	caption,
	dsfont,
	epstopdf,
	mathrsfs,
	mathabx,
	multirow,
	mathtools,
	dsfont,
	xcolor,
	tabu,
	subfig,
	hyperref}
\hypersetup{
    colorlinks=true,
    allcolors=blue
}
\definecolor{redy}{rgb}{0.84, 0.48, 0.47}
\definecolor{greeny}{rgb}{0.45, 0.68, 0.49}
\definecolor{bluey}{rgb}{0.42, 0.6, 0.78}
\newcommand{\acts}{\lefttorightarrow}
\newcommand{\Aut}{\mathrm{Aut}}

\newcommand{\D}{\mathrm{D}}
\newcommand{\dist}{\mathrm{d}}
\newcommand{\Dir}{\mathscr{D}}

\newcommand{\F}{\mathcal{F}}

\newcommand{\fu}{\varepsilon}
\newcommand{\g}{\gamma}
\newcommand{\G}{\Gamma}
\newcommand{\Gal}{\mathfrak{G}}

\newcommand{\HtH}{\Hyp^2\times\Hyp^2}

\newcommand{\Hyp}{\mathcal{H}}

\newcommand{\M}{\mathrm{M}}

\newcommand{\Mob}{M\"{o}bius }
\newcommand{\NN}{\mathbb{N}}
\newcommand{\PSL}{\mathrm{PSL}}
\newcommand{\QQ}{\mathbb{Q}}
\newcommand{\R}{\mathscr{R}}
\newcommand{\RR}{\mathbb{R}}

\newcommand{\T}{\mathrm{T}}
\newcommand{\Tsec}{\mathcal{T}}
\newcommand{\tr}{\mathrm{tr}}
\newcommand{\U}{\mathrm{U}}

\newcommand{\ZZ}{\mathbb{Z}}

\newtheorem{thm}{Theorem}[section]
\newtheorem{lem}[thm]{Lemma}
\newtheorem{prop}[thm]{Proposition}
\newtheorem{cor}[thm]{Corollary}

\newtheorem{defn}[thm]{Definition}
\newtheorem{rem}[thm]{Remark}

\title[Cusp Shapes of Hilbert-Blumenthal Surfaces]
	{Cusp Shapes of Hilbert-Blumenthal Surfaces}
\date{\today}

\newcommand{\Address}{{
\begin{center}
   \textsc{{\bf Joseph Quinn}\\
	\email{josephanthonyquinn@gmail.com}\\
	National Museum of Mathematics\\
	134 W 26th Street, Suite 4-S\\
	New York, NY 10001\\
	USA\\
	\vspace{12pt}	
	{\bf Alberto Verjovsky}\\
 	\email{alberto@matcuer.unam.mx}\\
	Instituto de Matem\'aticas, Unidad Cuernavaca\\
 	Universidad Nacional Aut\'onoma de M\'exico\\
	Av Universidad s/n, Col. Lomas Chamilpa\\
	CP 62210, Cuernavaca, Morelos\\
	M\'exico}\\
	\vspace{12pt}
 \end{center}
}}

\begin{document}

\maketitle
\Address

\begin{abstract}
	We introduce a new fundamental domain $\R_n$
		for a cusp stabilizer of a Hilbert modular group $\G$
		over a real quadratic field $K=\QQ(\sqrt n)$.
	This is constructed as the union
		of Dirichlet domains for the maximal unipotent group,
		over the leaves in a foliation of $\Hyp^2\times\Hyp^2$.
	The region $\R_n$
		is the product of $\RR^+$
		with a $3$-dimensional
		tower $\Tsec_n$
		formed by deformations of lattices
		in the ring of integers $\ZZ_K$,
		and makes explicit the cusp cross section's Sol $3$-manifold
		structure and Anosov diffeomorphism.
	We include computer generated images and data
		illustrating various examples.
\end{abstract}

\section{Introduction}

A Hilbert-Blumenthal group is some $\G=\PSL_2(\ZZ_K)$
	where $\ZZ_K$
	is the ring of integers of a real quadratic field $K$,
	and a Hilbert-Blumenthal surface is a quotient $\M_\G=(\HtH)/\G$
	of the product of two hyperbolic upper half-planes $\Hyp^2$
	by the \Mob action of $\G$
	under Galois conjugation.
As a generalization of modular curves,
	$\M_\G$
	represents the moduli space of Abelian varieties
	with real multiplication
	by $\ZZ_K$ \cite{Mcmullen2007}
	and these complex surfaces are a prototype for Shimura varieties,
	placing them at an interesting juncture of geometry,
	topology and number theory.
	
Here we are motivated by the search for a fundamental domain
	for the action of $\G$
	on $\HtH$
	which accurately reflects the geometry of $\M_\G$,
	a topic that dates back to Blumenthal
	\cite{Blumenthal1903,Blumenthal1904}.
Historically,
	our understanding of such a domain
	has improved with our understanding of its cusps.
Maass \cite{Maass1940}
	showed that the number of cusps
	equals the class number of $K$,
	then Siegel \cite{Siegel1961}
	computed a fundamental domain
	as a union over one piece at each cusp using an alternative metric.
This yields a complex surface with quotient singularities
	and cusp singularities,
	which Hirzebruch \cite{Hirzebruch1971}
	showed how to smoothly compactify
	(see also \cite{Vandergeer2012}
	and \S21 of \cite{BarthEtc2015}).
While these advances have been fruitful in understanding
	arithmetic and topological properties,
	certain geometric properties had remained elusive.
A cusp (cross) section of $\M_\G$
	is a $3$-dimensional mapping torus
	of some Anosov diffeomorphism $\varphi$
	of the torus and,
	due to McReynolds \cite{Mcreynolds2006,Mcreynolds2008},
	every Sol $3$-manifold
	is commensurable to one of these cusp sections
	up to diffeomorphism.
However,
	there had previously been no combinatorial description
	of these in terms of their sides and the action of $\varphi$
	as an explicit side-pairing map.
Here we provide this,
	subsequently also modeling an example from every
	commensurability class of the Sol $3$-manifolds.
	
To do this,
	we weaken the product metric on $\HtH$
	to a semimetric $\delta$,
	which restricts to a scaled Euclidean metric on each leaf
	of a natural foliation of the space.
This allows us to build a fundamental domain $\R_n$
	for the maximal unipotent subgroup of the cusp stabilizer
	as a union of toroidal Dirichlet domains in the leaves
	with respect to $\delta$.
This region
	is a $3$-dimensional
	hypersurface $\Tsec_n$
	(modeling the cusp section $\T^3_\varphi$)
	crossed with $\RR^+$,
	up to uniformly scaling the metric.
The shape of $\Tsec_n$
	is described by lattices from the ring of integers $\ZZ_K$
	of $K$
	that deform along a common axis of symmetry,
	in a way that is effectively computable from the fundamental unit $\fu$
	of $\ZZ_K$.
The Anosov diffeomorphism is then the diagonal matrix
	with diagonal $(\fu,\fu^{-1})$,
	which glues one lattice slice
	to another.
We also provide a map $\Psi_n:\HtH\rightarrow\RR^3$
	with which one can plot the image of $\Tsec_n$
	to visualize the cusp section
	(see Figure \ref{fig:towers}).	
 
\section{Preliminaries}\label{sec:prelim}

\subsection{Hilbert-Blumenthal surfaces and cusp sections}
\label{sec:prelim sub:HBS}

Let $\Hyp^2$
	be the upper half-plane model for the hyperbolic plane
	and denote the usual metric by $\dist_{\Hyp^2}$.
We will be interested in the product space $\HtH$,
	in which the points are of the form $(x_1+y_1i,x_2+y_2i)$
	where $x_1,x_2\in\RR$
	and $y_1,y_2\in\RR^+$,
	and we fix this notation throughout.
For $\g=
	\begin{pmatrix}
		a & b\\
		c & d
	\end{pmatrix}\in\PSL_2(\RR)$
	and $p\in\Hyp^2$,
	let $\g(p)$
	denote the usual isometric action by \Mob transformations,
	$$\g(p)=\frac{ap+b}{cp+d}.$$
	
Let $K$
	be a real quadratic number field
	and let $\sigma$
	be the non-trivial element of the Galois group $\Gal(K:\QQ)$.
That is,
	$K=\QQ(\sqrt n)$
	for some square-free $n\in\NN$
	and $\forall a,b\in\QQ$,
	we have $\sigma(a+b\sqrt n)=(a-b\sqrt n)$.
Let $\ZZ_K$
	be the ring of integers of $K$,
	i.e. $\ZZ_K=\ZZ\oplus\ZZ\alpha$
	where
	$$\alpha:=
		\begin{cases}
			\sqrt n \quad&;\quad n\not\equiv_41\\
			\frac{1+\sqrt n}{2} \quad&;\quad n\equiv_41
		\end{cases}.$$	
Let $\G:=\PSL_2(\ZZ_K)$
	and for $\g\in\G$
	let 
	then $\sigma(\g)$
	denote the application of $\sigma$
	to the entries of $\g$.
Then
	$$\G\acts
		\HtH:
		\quad \g(p_1,p_2)=
			\big(\g(p_1),\sigma(\g)(p_2)\big)$$
	is a discrete action by isometries.

\begin{defn}\samepage\label{defn:cusp}
	\begin{enumerate}
		\item[]
		\item We call $\G$
				a \emph{Hilbert-Blumenthal group},
				and we call the orbifold $(\HtH)/\G$
				a \emph{Hilbert-Blumenthal surface},
				which we denote by $\M_\G$.
		\item For $p\in\partial(\HtH)$,
				the \emph{stabilizer of $p$ in $\G$},
				is $\Delta_\G(p):=\big\{\g\in \G \;\big|\; \g(p)=p\big\}$.
		\item The \emph{maximal unipotent subgroup (of $\G$
				at $p$)},
				denoted by $\U_\G(p)$,
				is the group of all unipotent elements of $\Delta_\G(p)$.
		\item When $\U_\G(p)\neq\O$,
				we say that that $\G$
				(or equivalently that $(\HtH)/\G$)
				has a \emph{cusp}
				at $p$,
				and in this case we call $\Delta_\G(p)$
				a \emph{cusp group}.
	\end{enumerate}
\end{defn}

\begin{rem}
	For $\g\in\Delta_\G(p)$,
		the condition that $\g\in\U_\G(p)$
		is equivalent to saying that $\g$
		has a unique fixed point as an isometry of
		$\Hyp^2\cup\partial\Hyp^2$,
		which lies in $\partial\Hyp^2$,
		i.e. that $|\tr(\g)|=2$.
\end{rem}

Every cusp group $\Delta_\G(p)$
	is conjugate in $\PSL_2(K)$
	to $\Delta_\G(\infty,\infty)$
	\cite[\S5.1]{Mcreynolds2006}.
Thus we take $p=(\infty,\infty)$,
	abbreviate $\Delta:=\Delta_\G(p)$
	and $\U:=\U_\G(p)$,
	and this incurs no loss of generality in discussing the cusp shape.
We denote $(\HtH)/\Delta$
	by $\M_\Delta$
	and observe that in a small neighborhood of the cusp,
	$\M_\G$
	and $\M_\Delta$
	coincide.
Such a neighborhood is called a \emph{cusp end},
	defined up to homeomorphism.

Matrices in $\Delta$
	are upper triangular,
	forcing their diagonal entries to be in the unit group $\ZZ_K^\times$.
But since $K$
	is a real quadratic field,
	$\ZZ_K^\times=\{\pm\fu^\ell\;|\;\ell\in\ZZ\}$
	where $\fu$
	is the \emph{fundamental unit}
	of $\ZZ_K$,
	defined by $\fu:=\text{min}\{z\in\ZZ_K^\times\;|\; z>1\}$
	(see \cite{Azuhata1984}
	for additional characterizations).
Thus we have
	$$\Delta=\left\{
		\begin{pmatrix}
			\fu^\ell & z\\
			0 & \fu^{-\ell}
		\end{pmatrix}
		\;\middle|\;
		\ell\in\ZZ,\ z\in\ZZ_K\right\}$$
	up to $\pm1$,
	recalling that opposite signs are identified in $\PSL_2(\RR)$.

Let $\tau_z:=
		\begin{pmatrix}
			1 & z\\
			0 & 1
		\end{pmatrix}$
	where $z\in\ZZ_K$.
Then $\forall z\in\ZZ_K$,
\begin{align}\label{tau}
	\tau_z(x_1+y_1i,x_2+y_2i)
		=\left(\big(x_1+z\big)+y_1i,\big(x_2+\sigma(z)\big)+y_2i\right),
\end{align}
	affecting only the real parts of the points.
A computation using the trace shows that
	$\U=\{\tau_z\;|\; z\in\ZZ_K\}$,
	hence $\U=\langle\tau_1,\tau_\alpha\rangle$.
	
Let $\eta_\ell:=
	\begin{pmatrix}
			\fu^\ell & 0\\
			0 & \fu^{-\ell}
	\end{pmatrix}$,
	then
\begin{align}\label{eta}
	\eta_\ell(x_1+y_1i,x_2+y_2i)
		=\big(\fu^{2\ell}(x_1+y_1i),\fu^{-2\ell}(x_2+y_2i)\big).
\end{align}	
Let $\D:=\{\eta_\ell\ \;|\; \ell\in\ZZ\}$,
	then $\D=\langle\eta_1\rangle$
	and $\Delta=\langle\tau_1,\tau_{\alpha},\eta_1\rangle$.
The full Hilbert-Blumenthal group is attained by including,
	in the generators,
	the element
	$\iota:=\begin{pmatrix}
		0 & 1\\
		-1 & 0
	\end{pmatrix}$,
	an inversion through the unit hemisphere in each factor.
That is $\G=\langle\tau_1,\tau_\alpha,\eta_1,\iota\rangle$.

The cusp group $\Delta$
	admits a semidirect product decomposition,
	as follows.
The group $\U$
	is a normal subgroup of $\Delta$
	and since $\ZZ_K\cong\ZZ\oplus\ZZ\alpha$
	as an additive group,
	$\U\cong\ZZ^2$.
Also,
	$\D\cong\ZZ$
	is a cyclic subgroup of $\Delta$
	and $\U$
	is invariant under conjugation by $\D$,
	in particular
	$\eta_\ell\cdot\tau_z\cdot\eta_{-\ell}=\tau_{2\ell z}$.
This action by conjugation defines a homomorphism
	$\D\rightarrow\Aut(\U)$,
	giving
\begin{align}\label{semidirectUD}
	\Delta\overset{\cong}{\longrightarrow}\U\rtimes\D:\quad
		\begin{pmatrix}
			\fu^\ell & z\\
			0 & \fu^{-\ell}
		\end{pmatrix}
		\mapsto
		(\tau_z,\eta_\ell).
\end{align}
This admits the topological interpretation that $\M_\Delta$
	is diffeomorphic to $\rm{T}^3_\varphi\times\RR^+$,
	where $\rm{T}^3_\varphi$
	is the infrasolv manifold (in the sense of \cite[\S2.4.3]{Mcreynolds2006})
	that fibers over the circle,
	with fiber the torus,
	and Anosov diffeomorphism $\varphi$
	\cite{Mcreynolds2008}.
We call $\rm{T}^3_\varphi$
	the \emph{cusp section}
	of $\M_\G$
	(or equivalently,
	of $\G$).
%

\subsection{Fundamental domains}\label{sec:prelim sub:FD}

A \emph{fundamental domain}
	for a group $G$
	acting on a topological space $\mathcal{X}$
	is a subspace of $\mathcal{X}$,
	which we denote by $R_G(\mathcal{X})$
	(or just $R_G$
	when $\mathcal{X}$
	is clear)
	such that $\underset{g\in G}{\bigcup}g(R_G)=\mathcal{X}$,
	and for all pairs $g,g'\in G$,
	the intersection $g(R_G)\cap g'(R_G)$
	has no interior.
Notice that this notation does not indicate any specific choice for the domain,
	and we will introduce different notation when we wish to indicate
	our particular construction.
	
Some aspects of fundamental domains
	$R_\G(\HtH)$
	have remained consistent since the classical approach
	while others have varied.
A common theme is the use of an intersection of some choices for
	$R_\U$,
	$R_\D$,
	and $R_{\langle\iota\rangle}$,
	to attain an initial approximation of the domain,
	formalized by G\"otzky \cite{Gotzky1928}
	and later termed a G\"otzky region \cite{Deutsch2010}.
Usually,
	this properly contains a fundamental domain for $\G$,
	and the boundary intersecting $R_{\langle\iota\rangle}$
	is difficult to describe
	(see Remark \ref{floors}).
However,
	$R_U\cap R_D$
	forms a true fundamental domain for the group $\Delta$
	due to the semidirect product structure $\Delta\cong\U\rtimes\D$.

\subsubsection{\bf \boldmath$R_\D$ and \boldmath$\R_\D$}
\label{sec:prelim sub:RD}

This aspect has remained consistent in the literature since Blumenthal 
	\cite{Blumenthal1903},
	and will be used here as well up to a minor alteration.
For each $y_1,y_2\in\RR^+$,
	let
\begin{align}\label{FH}
	\F(y_1,y_2)&:=
		\big\{(x_1+y_1i,x_2+y_2i)\;|\;x_1,x_2\in\RR\big\}
		\subset\HtH,
\end{align}
	the pair of horizontal lines at $y_1$
	in the first factor and at $y_2$
	in the second factor.
Observe that
\begin{align*}
	\underset{y_1,y_2\in\RR^+}{\bigsqcup}\F(y_1,y_2)=\HtH,
\end{align*}
	so $\F:=\big\{\F(y_1,y_2)\mid y_1,y_2\in\RR^+\big\}$
	foliates $\HtH$.
There is a natural bijection
\begin{align}\label{Pi}
	\Pi:\F\rightarrow\RR^+\times\RR^+,\quad
		\F(y_1,y_2)\mapsto(y_1,y_2).
\end{align}
By (\ref{eta}),
	$\D$
	permutes the leaves via
\begin{align}\label{etaF}
	\eta_\ell\left(\F(y_1,y_2)\right)=\F(\fu^{2\ell}y_1,\fu^{-2\ell}y_2),
\end{align}
	thus in the image under $\Pi$,
	$\D$
	preserves each hyperbola in the set
	$\left\{y_1y_2=c\;\middle|\; y_1,y_2\in\RR^+\right\}_{c\,\in\RR^+}$,
	which foliates $\HtH$
	under $\Pi^{-1}$.
	
A natural fundamental domain for the action of $\D$
	is thus obtained as the image under $\Pi^{-1}$
	of the wedge between a pair of rays approaching the origin,
	identified by $\eta_1$.
We differ from the classical approach in how we choose the pair of rays,
	(justified in \S\ref{sec:CS}),
	and define our fundamental domain for $\D$
	as
\begin{align}\label{RD}
	\R_\D:=\Pi^{-1}\left(\big\{(y_1,y_2)\in\RR^+\times\RR^+\;\big|\;
		y_2\leq y_1<\fu^4y_2\big\}\right).
\end{align}

\begin{figure}[!htb]
	\centering
	\includegraphics[scale=.45]{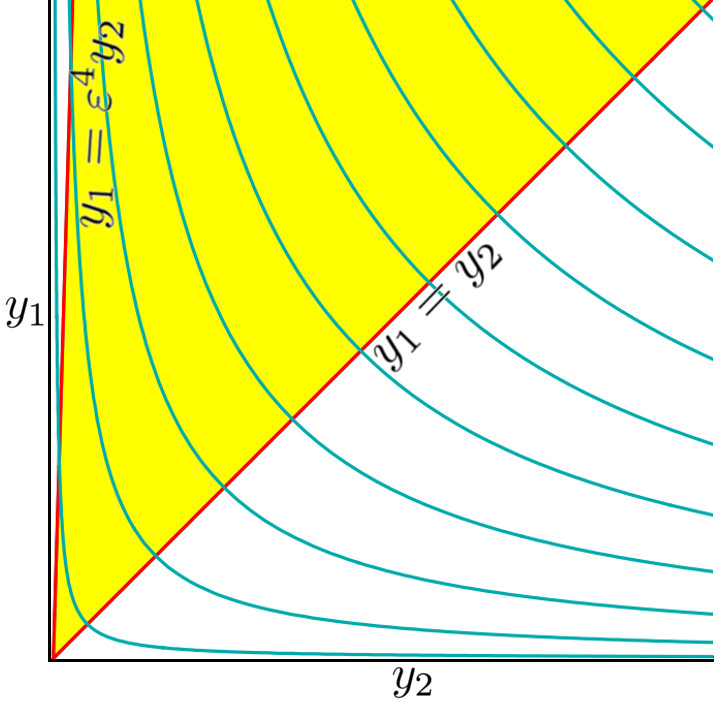}
	\caption{\cite{Mathematica}
		The fundamental domain $\R_\D$
			is shown (in yellow),
			for the action of $\D$
			on $\F$,
			along with
			a $\D$-invariant
			foliation by hyperbolas (in cyan).}
	\label{fig:FDD}
\end{figure}

\subsubsection{\bf \boldmath$R_\U$}
\label{sec:prelim sub:FD sub:U}

By (\ref{tau}),
	$\U$
	fixes each leaf of the foliation $\F$,
	and since $\Delta=\U\rtimes\D$,
	we have that $R_\U\cap\R_\D$
	is a fundamental domain for $\Delta$
	regardless of one's choice of $R_\U$.
Topologically,
	each quotient $\F(y_1,y_2)/\U$
	is a flat torus and $\F/\U$
	is foliated by these tori.
If one is interested in arithmetic properties of $\M_\G$
	as a topological manifold,
	one can represent orbits of $\U$
	as Siegel \cite{Siegel1961}
	does,
	using a reduction with respect to the field norm on $K$,
	but we are interested in a more geometrically accurate description
	which,
	in the following ways,
	resembles the classical approach.
	
Define the \emph{height}
	of a leaf $\F(y_1,y_2)$
	in the foliation $\F$,
	or of a point $(x_1+y_1i,x_2+y_2i)$
	in the leaf,
	as the product $y_1y_2$.
A set of points at some fixed height corresponds to
	the image under $\Pi^{-1}$
	of a hyperbola in Figure \ref{fig:FDD}.
Also,
	under $\Pi^{-1}$,
	a set of tori at the piece
	of a hyperbola between the rays $y_1=y_2$
	and $y_1=\fu^4y_2$
	gives a geometric representation of the cusp section
	$\T^3_\varphi$,
	and the tori at the entire wedge between these rays
	gives a diffeomorphic representation of the cusp end.

Define the \emph{level}
	of a ray as in Figure \ref{fig:FDD},
	or of a point $(x_1+y_1i,x_2+y_2i)$
	in a leaf on this ray,
	as the quotient $y_1/y_2$.
There is a natural bijection at each leaf
\begin{align}\label{pi}
	\pi_{y_1,y_2}:\F(y_1,y_2)\rightarrow\RR^2,\quad
		(x_1+y_1i,x_2+y_2i)\mapsto(x_1,x_2).
\end{align}
Under these bijections,
	points at the same level (varying the height)
	correspond to spaces where the metric
	scales uniformly,
	and points at the same height (varying the level)
	correspond to spaces where the metric
	expands along one axis
	and contracts along another.	

The classical approach to $R_\U$
	is to use the same rectangle for every torus in the foliation,
	making $R_\U$
	an infinitely tall parallelepiped.
While this gives a straightforward fundamental domain for $\Delta$
	in analogy to the classical cusp group of $\PSL_2(\ZZ)$,
	it does not account for the change in metric that occurs
	in the leaves as the level varies,
	and obscures the action of $\varphi$
	as a side-pairing upon passing to the quotient $\T^3_\varphi$.
We will represent (in \S\ref{sec:lattices})
	the orbit of $\U$
	by tori that change shape with the deforming metric
	as the leaves vary,
	allowing us to write $\varphi$
	explicitly.	

\begin{rem}\label{floors}
	The classical choice for $R_{\langle\iota\rangle}$
		is  $\big\{p\in\HtH\;\big|\;|p_1||p_2|\geq1\big\}$.
	Cohn
		\cite{Cohn1965shape,Cohn1965floors,Cohn1966,Cohn1969}
		studied how to
		approximate a fundamental domain for all of $\G$
		by intersecting this with the classical parallelepiped 
		model of $R_\U$.
	He found that in all but the case $K=\QQ(\sqrt5)$,
		this has an infinite-volume boundary
		due to $3$-dimensional
		regions approaching $\partial(\HtH)$.
	He introduced the notion of the ``floor''
		as an alternative representation of this boundary,
		having finite volume at the sacrifice of connectedness.
	The current authors explored applications to this
		of our new choice for $R_\U$,
		but found the slight improvements
		not worth the computational complexity.
\end{rem}

\subsubsection{\bf Dirichlet domains}\label{sec:prelim sub: FD sub:DD}

When $\mathcal{X}$
	is a metric space,
	we can use its metric to form a type of
	fundamental domain with additional geometric properties.

\begin{defn}\label{defn:dir}
	Let $G$
		be a group of isometries acting on a topological space
		$\mathcal{X}$,
		let $d_\mathcal{X}$
		be a metric on $\mathcal{X}$,
		and let $c\in\mathcal{X}$
		satisfy $\Delta_G(c)=\{1\}$.
	Then the \emph{Dirichlet domain
		for $G$
		with respect to $d_\mathcal{X}$,
		centered at $c$},
		is
		$$\Dir_c(G):=\Big\{x\in\mathcal{X}\;\Big|\;
			\forall\g\in\G,\
			d_\mathcal{X}(c,x)\leq d_\mathcal{X}\big(c,\g(x)\big)
			\Big\}.$$
\end{defn}

Then $\Dir_c(G)$
	is convex and tiles $\mathcal{X}$
	under the group action.
Each pair of sides of $\Dir_c(G)$
	is contributed by an isometry and its inverse,
	which are identified by that isometry under the group action.
The set of isometries that contribute the sides of $\Dir_c(G)$
	generate the group $G$,
	and when $G$
	is finitely generated,
	so will be the number of sides.
	\cite{Engel1986}
	
We can identify $\Dir_c(G)$
	and its sides with the following tools.

\begin{defn}\samepage\label{defn:mediatriz}
	Let $\mathcal{X}$,
		$G$
		and $c$
		be as in Definition \ref{defn:dir}.
	Let $p,q\in\mathcal{X}$
		with $p\neq q$.
	The \emph{semispace contributed (to $\Dir_c(G)$)
		by $g$
		(with respect to $d_\mathcal{X}$)}
		is
		$$E_c(g):=\big\{x\in\mathcal{X}\;\big|\;
			d_\mathcal{X}(x,c)\leq
			d_\mathcal{X}\big(g(x),c\big)\big\},$$
		and the \emph{mediatrize contributed
		(to $\Dir_c(G)$) by $g$
		(with respect to $d_\mathcal{X}$)}
		is the set of points at equality.
\end{defn}

Thus,
	$$\Dir_c(G)=\underset{g\in G\smallsetminus\{1\}}
		{\bigcap}E_c(g)$$
	and,
	since $m_c(g)=\partial E_c(g)$,
	each side of $\Dir_c(G)$
	is a portion of a mediatriz.
A convenient characterization of $m_c(g)$
	is as the set of points equidistant from $c$
	and $g^{-1}(c)$.
To see this,
	take the defining equation for $m_c(g)$
	and apply the isometry $g^{-1}$
	to the arguments of the distance function on the right hand side.
That is,
\begin{align}\label{m}
	m_c(g)=\big\{x\in\mathcal{X}\;\big|\;
		d_\mathcal{X}(x,c)=
		d_\mathcal{X}\big(x,g^{-1}(c)\big)\big\}.
\end{align}	

\begin{rem}\samepage
	We prefer the Spanish term ``mediatriz''
		(plural: ``mediatrices'')
		to the more common term ``perpendicular bisector,''
		usually defined with respect to the geodesic from $c$
		to $g^{-1}(c)$.
	Our reason is that a pair of distinct points in $\HtH$
		does not have a unique geodesic connecting them,
		since its metric is the $\ell^1$
		sum over the $\Hyp^2$
		metrics,
		similarly to how the Manhattan metric
		does not give unique geodesics.
\end{rem}

Our main result,
	Theorem \ref{thm:FDD},
	is phrased in terms of notation constructed throughout the article,
	in such a way that the proof is provided by the content
	up to the statement of the theorem.
The idea is to create a geometrically accurate fundamental domain
	for $\Delta$,
	and give an algorithm to find
	the sides of each cusp section $\T_\varphi^3$
	along with their gluing maps,
	including the Anosov diffeomrophism.
	
\section{Lattice Deformations}\label{sec:lattices}

In this section,
	we construct a fundamental domain for $\U$
	as a union of Dirichlet domains for its action
	on the leaves of the foliation $\F$.
Consider the function
\begin{align}\label{delta}
\begin{split}
	\delta:(\HtH)\times(\HtH)
		\rightarrow&\ \RR^{\geq0},\\
	\big((p_1,p_2),(q_1,q_2)\big)\mapsto&\
		\frac{|p_1-q_1|^2}{\Im(p_1)\Im(q_1)}
			+\frac{|p_2-q_2|^2}{\Im(p_2)\Im(q_2)}.
\end{split}
\end{align}
We will use $\delta$
	to describe the sides of our fundamental domain
	as solution sets to cubic polynomials.
	
\begin{rem}
	By way of motivation,
		$\delta$
		is a simplification of the standard metric on $\HtH$
		formed by eliminating reliance on transcendental functions.
	This grants us more manageable computations
		at the sacrifice of the triangle inequality.
	That is,
		$\delta$
		is not a metric on $\HtH$,
		but a semi-metric.
	Lemma \ref{lem:delta}
		sets us up to use $\delta$
		as desired regardless,
		and its proof illustrates the specific relationship between $\delta$
		and the standard metric.
\end{rem}

For each $y_1,y_2\in\RR^+$,
	let
	$$\delta_{y_1,y_2}:=\delta|_{\F(y_1,y_2)\times\F(y_1,y_2)}.$$
For part (3) below,
	recall from \S\ref{sec:prelim sub:HBS}
	the notation $\tau_z=
		\begin{pmatrix}
			1 & z\\
			0 & 1
		\end{pmatrix}$
	where $z\in\ZZ_K$,
	and that elements of this form comprise the group $\U$.
Also,
	we abbreviate $E_{(y_1i,y_2i)}$
	by $E_{y_1,y_2}$,
	and $m_{(y_1i,y_2i)}$
	by $m_{y_1,y_2}$,
	and these always denote subsets of the leaf $\F(y_1,y_2)$.
	
\begin{lem}\samepage\label{lem:delta}
	\begin{enumerate}
		\item[]
		\item
			The function $\delta$
				is invariant under the action of $\PSL_2(\RR)$.
		\item
			For each $y_1,y_2\in\RR^+$,
				the restriction
				$\delta_{y_1,y_2}$
				is a metric on the leaf $\F(y_1,y_2)$.
		\item	
			Within the leaf $\F(y_1,y_2)$
				and with respect to the metric $\delta_{y_1,y_2}$,
				the semispace $E_{y_1,y_2}(\tau_z)$
				is the solution set to
			\begin{align*}
				\frac{2x_1z+z^2}{y_1^2}
					+ \frac{2x_2\sigma(z)+\sigma(z)^2}{y_2^2}
					\geq0
			\end{align*}
				and the mediatriz $m_{y_1,y_2}(\tau_z)$
				is the set of points at equality.
	\end{enumerate}
\end{lem}

\begin{proof}
	The standard distance formula on $\Hyp^2$
		is
		$$d_{\Hyp^2}(w,z)=\log\left(x+\sqrt{x^2-1}\right)$$
		where
	\begin{align*}
		x&=1+\frac{|z-w|^2}{2\Im{w}\Im{z}},
	\end{align*}
		so all dependence of $d_{\Hyp^2}(w,z)$
		on $w$
		and $z$
		occurs in the term $\dfrac{|z-w|^2}{\Im{w}\Im{z}}$.
	Since $d_{\Hyp^2}$
		is invariant under the action of $\PSL_2(\RR)$,
		so is this term,
		and since $\delta$
		the sum over such terms,
		$\delta$
		inherits that property as well.

	To prove part (2),
		recall the definition of $\F(y_1,y_2)$
		from (\ref{FH}),
		and let $p,q\in\F(y_1,y_2)$.
	Then $\exists x_1, x_2, x_3, x_4\in\RR$
		such that $p=(x_1+iy_1, x_2+iy_2)$
		and $q=(x_3+iy_1, x_4+iy_2)$,
		and,
		using (\ref{delta}),
		we compute
	\begin{align}\label{leafmetric}
		\delta_{y_1,y_2}(p,q)
		&=y_1^{-2}|x_1-x_3|^2+y_2^{-2}|x_2-x_4|^2.
	\end{align}
	Each summand is a metric on $\RR$,
		obtained by scaling the squared Euclidean metric
		by a positive constant.
	The result follows because the sum is the $\ell^1$
		metric over these on $\RR^2
		=\pi_{y_1,y_2}\big(\F(y_1,y_2)\big)$.
		
	For part (3),
		recalling Definition \ref{defn:dir},
		$E_{y_1,y_2}(\tau_z)$
		is the solution set in $\F(y_1,y_2)$
		to the inequality
		$$\delta_{y_1,y_2}\big((x_1+y_1i,x_2+y_2i),(y_1i,y_2i)\big)
			\leq\delta_{y_1,y_2}
			\big(\tau_z(x_1+y_1i,x_2+y_2i),(y_1i,y_2i)\big),$$
		and $m_{y_1,y_2}(\tau_z)$
		is the set of points at equality.
	Carrying out the action of $\tau_z$,
		the right hand side of this is
		$\delta\big(x_1+z+y_1i,x_2+\sigma(z)+y_2i),(y_1i,y_2i)\big)$
		where $\sigma$
		is the non-trivial element of the Galois group $\Gal(K:\QQ)$
		(and observe that $x_1+z,x_2+\sigma(z)\in\RR$).
	Next,
		rewrite each side of the inequality
		according to (\ref{leafmetric}).
	Rearranging terms
		(notice that the $x_1^2$
		and $x_2^2$
		terms cancel)
		yields the desired formula.
\end{proof}

For each pair $y_1,y_2\in\RR^+$,
	let $\T(y_1,y_2):=\Dir_{(y_1i,y_2i)}(\U)$,
	the Dirichlet domain for the action of $\U$
	on the leaf $\F(y_1,y_2)$
	with respect to the metric $\delta_{y_1,y_2}$,
	centered at $(y_1i,y_2i)$.
As discussed in \ref{sec:prelim sub:FD sub:U},
	these $\T(y_1,y_2)$
	slices are flat tori which,
	ranging over $y_1,y_2\in\RR^+$,
	foliate $(\Hyp^2\times\Hyp^2)/\U$.
We choose our fundamental domain for $\U$
	to be the union of these slices,
\begin{align}\label{RU}
	\R_\U:=\underset{y_1,y_2\in\RR}{\bigsqcup}\T(y_1,y_2).
\end{align}
Similarly,
	we extend the semispaces and mediatrices
	from the leaves across $\HtH$,
	and introduce the following notation.
For $z\in\ZZ_K$,
	define
\begin{align*}
	E(z)&:=\underset{y_1,y_2\in\RR}{\bigsqcup}
		E_{y_1,y_2}(\tau_z),\text{ and}\\
	m(z)&:=\underset{y_1,y_2\in\RR}{\bigsqcup}
		m_{y_1,y_2}(\tau_z),
\end{align*}
	so that $E(z)\cap\F(y_1,y_2)=E_{y_1,y_2}(\tau_z)$
	and $m(z)\cap\F(y_1,y_2)=m_{y_1,y_2}(\tau_z)$.
Then some collection of the $m(z)$
	(over $z\in\ZZ_K\smallsetminus\{0\}$)
	form the sides of $\R_\U$,
	and
	(from Lemma \ref{delta}, part (3))
	each $m(z)$
	is an algebraic variety in $\HtH$
	defined by the cubic polynomial
\begin{align}\label{mediatrizey}
	2zx_1+z^2+
		\Big(\frac{y_1}{y_2}\Big)^2
		\left(2\sigma(z)x_2+\sigma(z)^2\right)&=0
\end{align}
	in the four variables $x_1,x_2\in\RR$,
	and $y_1,y_2\in\RR^+$.

The next lemma evokes the projections
	$\pi_{y_1,y_2}:\F(y_1,y_2)\rightarrow\RR^2$
	defined by (\ref{pi}).
We have $\pi_{y_1,y_1}(c)=(0,0)$,
	and the orbit of $(0,0)$
	under $\U$
	in this projection is the same for all $y_1,y_2$,
	yet $\pi_{y_1,y_2}\big(m_{y_1,y_2}(z)\big)$
	depends on $y_1$
	and $y_2$
	due to the scaling from one local metric to another.

\begin{lem}\phantomsection\label{lem:levels}\samepage
	\begin{enumerate}
		\item[]
		\item
			The projection $\pi_{y_1,y_2}\big(m_{y_1,y_2}(\tau_z)\big)
				\subset\RR^2$
				is a Euclidean line.
		\item
			The pair of lines $\pi_{y_1,y_2}\big(m_{y_1,y_2}(\tau_z)\big)$
				and $\pi_{y_1,y_2}\big(m_{y_1,y_2}(\tau_z')\big)$
				are parallel if and only if
				$\exists q\in\QQ$
				such that $z'=qz$.
		\item
			If $\dfrac{y_1}{y_2}=\dfrac{y_1'}{y_2'}$,
				then $\pi_{y_1,y_2}\big(\T(y_1,y_2)\big)=
				\pi_{y_1',y_2'}\big(\T(y_1',y_2')\big)$.
		\item
			For each $y_1,y_2\in\RR^+$,
				$\pi_{y_1,y_2}\big(\T(y_1,y_2)\big)$
				is symmetric about $(0,0)$,
				is either a parallelogram or a hexagon,
				and these deform continuously with $y_1,y_2$.
	\end{enumerate}
\end{lem}

\begin{proof}
	By Lemma \ref{lem:delta} (3),
		the mediatriz $m_{y_1,y_2}(\tau_z)\subset\F(y_1,y_2)$
		is the solution set to equation (\ref{mediatrizey})
		in the $(x_1,x_2)$-coordinates,
		and $\T(y_1,y_2)$
		is bounded by these mediatrices.
	Parts (1) and (3) follow immediately.
	Part (2) follows from the fact that $\forall z\in\ZZ_K$
		and $\forall q\in\QQ$,
		$\sigma(qz)=q\sigma(z)$.
	
	Since
		$\big\{\big(z,\sigma(z)\big)\;\big|\; z\in\ZZ_K\big\}
			\subset\RR^2$
		is discrete,
		there are finitely many lines
		$\pi_{y_1,y_2}\big(m_{y_1,y_2}(\tau_z)\big)$
		contributing sides to $\pi_{y_1,y_2}\big(\T(y_1,y_2)\big)$.
	Altering the level $y_1/y_2$
		in (\ref{mediatrizey})
		deforms these lines continuously and,
		since $\sigma(-z)=-\sigma(z)$,
		they occur in pairs $\pi_{y_1,y_2}\big(m_{y_1,y_2}(\pm z)\big)$
		arranged symmetrically about the origin.
	Since $\T(y_1,y_2)$
		is a Dirichlet domain for the action of $\U$
		on $\F(y_1,y_2)$,
		we know that $\pi_{y_1,y_2}\big(\T(y_1,y_2)\big)$
		is convex and tiles the plane via translational symmetry. 
	The only possible number of sides
		for a convex Euclidean polygon that does this are $3, 4$
		and $6$,
		but since $\T(y_1,y_2)$
		has order $2$
		rotational symmetry,
		the number of sides must be $4$
		or $6$.
\end{proof}

Lemma \ref{lem:levels}
	tells us that along each fixed height $y_1y_2$
	(a $3$-dimensional space along a hyperbola in Figure \ref{fig:FDD};
	see \S\ref{sec:prelim sub:FD sub:U}),
	the fundamental domain $\R_\U$
	has the same $3$-dimensional
	shape:
	a continuum of parallelograms and hexagons arranged
	symmetrically about a central axis.
We next gain an explicit description of this by controlling the
	distribution of the parallelograms,
	as well as which $z\in\ZZ_K$
	contribute them.

\begin{prop}\label{prop:algorithm}
	In the $3$-dimensional subspace
		of $\R_\U$
		at some fixed height,
		the parallelogram cross-sections are distributed discretely
		along the axis of symmetry in the following way.
	If $\T(y_1,y_2)$
		is a parallelogram whose sides
		are contributed by $\pm z, \pm z'\in\ZZ_K$
		(choosing $z,z'>0$),
		then $y_1/y_2=\sqrt{\dfrac{-zz'}{\sigma(zz')}}$.
	The next parallelogram as $y_1/y_2$
		increases is contributed by $\pm(z+z')$
		and whichever of the pairs $\pm z$
		or $\pm z'$
		that has smaller absolute value under $\sigma$.
\end{prop}

\begin{proof}
	Recall (from \S\ref{sec:prelim sub:FD sub:U}),
		the level of a point $(x_1+iy_1,x_2+iy_2)$
		is $y_1/y_2$,
		and we'll now denote this by $k$.
	By Lemma \ref{lem:levels} (3),
		it suffices to look at $\T(k,1)$.
	For $z\in\ZZ_K$,
		denote the line
		$\pi_{k,1}\big(m_{k,1}(z)\big)\subset\RR^2$
		by $l_z$.
	Now fix $z,z'\in\ZZ_K^+$
		so that $l_{\pm z}$
		and $l_{\pm z'}$
		bound a parallelogram cross section at a fixed height of $\R_\U$,
		and denote this parallelogram by $P$.
	Using (\ref{mediatrizey})
		to write the equations for $l_z$,
		$l_{z'}$,
		and $l_{z+z'}$,
		we compute that these have a common intersection point
		if and only if $k=\sqrt{\frac{-zz'}{\sigma(zz')}}$.
	Thus $l_{-(z+z')}$
		passes through the opposite corner of $P$
		and,
		by a similar computation,
		the lines $l_{\pm(z-z')}$
		pass the through other pair of opposite corners of $P$.
	By hypothesis,
		no lines contributed by $\ZZ_K$
		enter $P$
		and observe also that no others pass through its corners:
		indeed $z$
		and $z'$
		are independent,
		so that any other element of $\ZZ_K$
		could be written as $az+bz'$
		(with $a,b\in\ZZ$
		not both $\pm1$),
		and a similar computation shows that
		$l_{az+bz'}$
		cannot pass through these corners
		at this height.
		
	Now,
		increasing $k$
		slightly changes the slopes of the lines so that $l_{z+z'}$
		enters the (stretching)
		parallelogram bounded by $l_{\pm z}$
		and $l_{\pm z'}$,
		and $l_{z-z'}$
		moves away from it.
	Decreasing $k$
		slightly has the opposite effect.
	(Figure \ref{fig:chop}
		shows an example of this.)
	\begin{figure}[!htb]
		\centering
		\includegraphics[scale=.15]{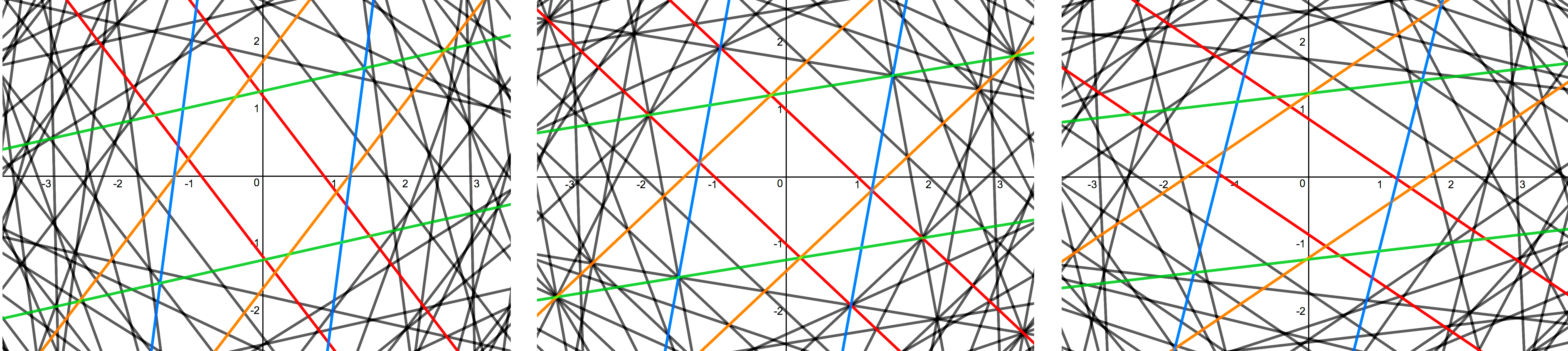}
		\caption{\cite{Desmos}
			Some mediatrices
				$\left\{m(\tau_z)\cap\F(k,1)\;\middle|\;
					z\in\ZZ_K\right\}$
				in the case $n=2$
				exemplify the situation described in the proof
				of Proposition \ref{prop:algorithm}.
			The mediatrices at height $y_1/y_2=1$
				are in the center,
				and those with $y_1/y_2$
				slightly decreased and increased
				are shown on the left and right,
				respectively.
			The pairs of mediatrices
				for $z=\pm1, \pm\sqrt n, \pm(1+\sqrt n)$
				and $\pm(1-\sqrt n)$
				are colored red, orange, blue, and green,
				respectively,
				and the mediatrices for all other $z\in\ZZ_K$
				that enter the shown region are shown in black.
			Note that when $n\equiv_41$,
				this will occur at heights other than $1$.}
		\label{fig:chop}
	\end{figure}
	Since this forms a hexagon,
		Lemma \ref{lem:levels}
		tells us that no other elements of $\ZZ_K$
		can simultaneously truncate more corners.
	Thus $P$
		is a parallelogram precisely at
		$k=\sqrt{\frac{-zz'}{\sigma(zz')}}$
		and,
		as the level increases,
		$l_{\pm(z+z')}$
		contributes a new pair of sides that
		persist until the next parallelogram is formed.
	These parallelograms are discretely distributed
		due to the discreteness of
		$\big\{\big(z,\sigma(z)\big)\;\big|\;z\in\ZZ_K\big\}$.
	
	Lastly,
		again analyzing the slopes of the lines as the level $k$
		increases,
		we see that $l_{\pm z}$
		move outside of $P$
		if and only if $|\sigma(z)|<|\sigma(z')|$,
		so the one with lower absolute value under $\sigma$
		persists in contributing a boundary as the other ceases to do so.
\end{proof}

\section{The Cusp Shape}
\label{sec:CS}

Let $\R_\Delta:=\R_\U\cap\R_\D$,
	our choice of fundamental domain for the cusp group $\Delta$.
We will also sometimes write $\R_n$
	for $\R_\Delta$
	to specify that $K=\QQ(\sqrt n)$
	(and $\Delta$
	is the stabilizer of $(\infty,\infty)$
	in $\PSL_2(\ZZ_K)$).
In this section,
	we state our main result,
	Theorem \ref{thm:FDD},
	deriving a precise description of $\R_n$
	and providing an algorithm for computing it effectively
	from the fundamental unit $\fu$
	of $\ZZ_K$.
	
As we saw in the previous section,
	$\R_\U$
	is comprised of equally shaped hypersurfaces,
	one at each fixed height,
	which are identical up to uniform scaling of the local metrics.
Recall from \S\ref{sec:prelim sub:RD}
	that $\R_\D$
	is the wedge bounded by the hypersurface at $y_1=y_2$
	(at level $1$)
	and the hypersurface at $y_1=\fu^4 y_2$
	(at level $\fu^4$).
Let
\begin{align}
	\Tsec_n(h):=\underset{1\leq y<\fu^4}{\bigsqcup}\T(hy,h),
\end{align}
	and when $h$
		does not matter,
		we will simply write $\Tsec_n$.
Then $\Tsec_n$
	is a geometric model for the cusp section and we have
	$\R_n=\underset{h\in\RR^+}{\bigsqcup}\Tsec_n(h)$.

We will now compute the shapes of the torus slices of $\overline{\Tsec_n}$
	at levels ($y_1/y_2=$)
	$1$
	and $\fu^4$,
	and see how they glue together under the action of $\D$.
At level $1$,
	we take advantage of the fact that the metric $\delta_{y_1,y_2}$
	is squared Euclidean when $y_1=y_2$.
(This was the motivation for our choice of this
	as a boundary component of $\R_\D$
	in \S\ref{sec:prelim sub:RD}.)

\begin{lem}\label{lem:T11}
	\begin{enumerate}
		\item[]
		\item\label{lem:T11 part:not1}
			If $n\not\equiv_41$,
				then $\pi_{y,y}\big(\T(y,y)\big)$
				is a rectangle
				whose sides are contributed by
				$$\left\{\tau_z\mid z=\pm1,\pm\sqrt n\right\}.$$
		\item\label{lem:T11 part:1}
			If $n\equiv_41$,
				then $\pi_{y,y}\big(\T(y,y)\big)$
				is a hexagon
				whose sides are contributed by
				$$\left\{\tau_z\;\Big|\; z=\pm1,\;
					\pm\frac{1+\sqrt n}{2},\;
					\pm\frac{1-\sqrt n}{2}\right\}.$$
	\end{enumerate}
\end{lem}
	
\begin{proof}		
	For an element $z=a+b\sqrt n\in\ZZ_K$
		with $a,b\in\QQ$,
		and any $y_1,y_2\in\RR^+$,
		we have
		$$\big(\pi_{y_1,y_2}\circ\tau_z|_{\F(y_1,y_2)}\big)(y_1i,y_2i)
			=(a+b\sqrt n,a-b\sqrt n).$$
	Therefore the orbit of $\U$
		on $(y_1i,y_2i)$
		at $\F(y_1,y_2)$
		forms a rectangular lattice if $n\not\equiv_41$,
		and forms a hexagonal lattice if $n\equiv_41$,
		where in both cases the lattice has the diagonal
		lines of symmetry $x_1=\pm x_2$.
	Now let $y_1=y_2=y$.
	The mediatrices $m(\tau_z)\cap\F(y,y)$
		are the Euclidean perpendicular bisectors
		between the points $(0,0)$
		and $\big(z,\sigma(z)\big)$.
	Thus $\T(y,y)$
		is rectangular if $n\not\equiv_41$
		and hexagonal if $n\equiv_41$
		(see Figure \ref{fig:FH11lattices}).
\begin{figure}[!htb]
	\centering
	\includegraphics[scale=.45]{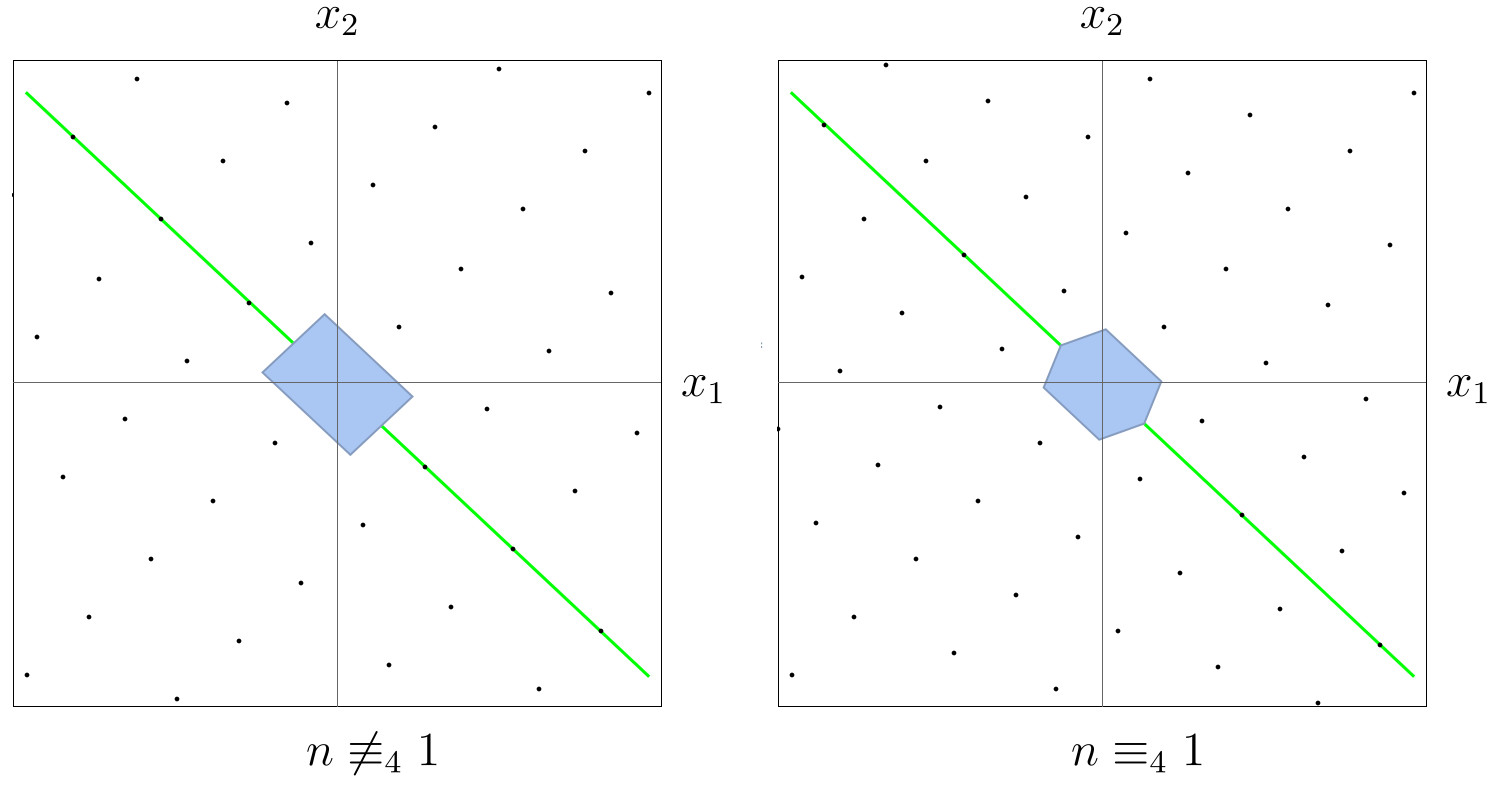}
	\caption{\cite{Mathematica}
		The points show the two lattice types formed by the orbit of $\U$
			on $(0,0)$
			under $\pi_{y_1,y_2}$,
			with the line of symmetry $x_1=-x_2$
			shown in green.
		The two types of toroidal Euclidean Dirichlet domains
			when $y_1=y_2$
			are shown in blue.}
	\label{fig:FH11lattices}
\end{figure}
		
	In both cases,
		sides are contributed by the orbit points closest to the origin,
		which,
		after projecting by $\pi_{y,y}$,
		are $\pm(1,1)$
		(contributed by $\tau_z$
		with $z=\pm1$).
	Additional sides are then contributed by the closest points to the origin
		that lie between the lines
		$$\pi_{y,y}\big(m_{y,y}(\tau_{\pm1})\big)
			=\big\{(x_1,x_2)\;\big|\;x_1+x_2=\pm1\big\}.$$
	When $n\not\equiv_41$,
		these are $\pm\left(\frac{\sqrt n}{2},-\frac{\sqrt n}{2}\right)$
		(contributed by $z=\pm\sqrt n$)
		and when $n\equiv_41$,
		they are $\pm\left(\frac{1+\sqrt n}{2},\frac{1-\sqrt n}{2}\right)$
		and $\pm\left(\frac{1-\sqrt n}{2},\frac{1+\sqrt n}{2}\right)$
		(contributed by $\tau_z$
			with $z=\pm\frac{1+\sqrt n}{2}$
			and $\pm\frac{1-\sqrt n}{2}$,
			respectively).
\end{proof}

To describe the torus in $\overline{\Tsec_n}$
	at level $\fu^4$,
	and also see how this attaches to the torus at level $1$
	under the action of $\D$,
	we will prove something more general.
Recall from (\ref{eta})
	that $\eta_\ell=
		\begin{pmatrix}
			\fu^\ell & 0\\
			0 & \fu^{-\ell}
		\end{pmatrix}$,
	that $\eta_\ell(x_1+iy_1,x_2+iy_2)
		=\big(\fu^{2\ell}(x_1+iy_1),\fu^{-2\ell}(x_2+iy_2)\big)$,
	and that $\Delta=\langle\eta_1\rangle$.

\begin{lem}\label{lem:etaT11}
	The image under $\eta_{1}$
		of $m_{y,y}(\tau_z)$
		is $m_{\fu^2y,\fu^{-2}y}(\tau_{\fu^2z})$.
\end{lem}

\begin{proof}
	The equation defining the mediatriz $m_{\fu^2y,\fu^{-2}y}(\tau_{z^2})$
		is obtained by taking (\ref{mediatrizey})
		and substituting $\fu^2z$
		for $z$,
		and $\fu^4$
		for $y_1/y_2$.
	Recalling that $\sigma(\fu)^2=\fu^{-2}$,
		this yields
%
%
%
		$$2zx_1+\fu^2z^2+
			2\fu^{4}\sigma(z)x_2+\fu^{2}\sigma(z)^2=0.$$		
	Applying $\eta_{-1}$($=\eta_1^{-1}$)
		to this line has the effect of sending it to the leaf $\F(y,y)$
		and altering its defining equation by appending a factor of $\fu^2$
		to $x_1$,
		and a factor of $\fu^{-2}$
		to $x_2$,
		after which all the $\fu$
		cancel,
		leaving the equation for $m_{y,y}(\tau_z)$.	
	Thus $\eta_{1}^{-1}\big(m_{\fu^2y,\fu^{-2}y}(\tau_{\fu^2z})\big)
			=m_{y,y}(\tau_z)$.
\end{proof}

Since $\eta_1$
	is an isometry of $\HtH$,
	Lemma \ref{lem:etaT11},
	implies a one-to-one correspondence
	between the sides of $\T(\fu^2y,\fu^{-2}y)$
	and the sides of $\T(y,y)$,
	as follows:
	$\tau_z$
	contributes a side to $\T(y,y)$
	if and only if $\tau_{z^2}$
	contributes a side to $\T(\fu^2y,\fu^{-2}y)$.
Moreover,
	$\eta_1\big(\T(y,y)\big)=\T(\fu^2y,\fu^{-2}y)$,
	so that $\eta_1$
	is the Anosov diffeomorphism that provides the final side-pairing
	on $\Tsec_n$.

Combining this with the other results thus far completes
	the geometric description of each fundamental domain $\R_n$
	(over all squarefree $n\in\NN$),
	and gives a straightforward algorithm for computing it.	
Before we collect these results into our theorem,
	we provide a map that will make the description more precise
	and facilitate computations.
Let
\begin{align*}
	\Psi:&\ \HtH\rightarrow\RR^3,\quad
		(x_1+y_1i,x_2+y_2i)\mapsto
		(x_1,x_2,\tfrac{y_1}{y_2}),
\end{align*}
	and write $k=y_1/y_2$
	(the level,
	as before)
	for the third coordinate.

\begin{thm}\label{thm:FDD}
	For $K=\QQ(\sqrt n)$
		($n$ squarefree),
		a cusp group of the Hilbert-Blumenthal surface
		$(\HtH)/\PSL_2(\ZZ_K)$
		admits a fundamental domain in which each cusp section
		$\Tsec_n$
		is described as follows.
		
\begin{center}
	\textbf{\textit Sides}
\end{center}
	The region $\Psi(\Tsec_n)$
		lies between the planes at $k=1$
		and $k=\fu^4$.
	The rest of the sides of $\Psi(\Tsec_n)$
		are the surfaces
		defined by the cubic polynomials
		$$2zx_1+z^2+k^2\big(2\sigma(z)x_2+\sigma(z)^2\big)
			=0$$
		contributed over the $z\in\ZZ_K$
		that occur in the list $L$
		constructed by the following algorithm.
	\begin{enumerate}
		\item
			Let $z_1=1$.
		\begin{enumerate}
			\item
				If $n\not\equiv_41$,
					let $z_1'=\sqrt n$.
			\item
				If $n\equiv_41$,
					let $z_1'=\frac{1+\sqrt{n}}{2}$.
		\end{enumerate}
		\item
			Let $z_m=\fu^2z_1$
				and $z_m'=\fu^2z_1'$.
		\item
			Start with $i=1$,
				and while $\{z_i,z_i'\}\neq\{z_m,z_m'\}$,
				let $z_{i+1}=z_i+z_i'$
				and let $z_{i+1}'$
				be the element of $\{z_i,z_i'\}$
				that is minimal under $|\sigma|$.
		\item
			Let $L=\{\pm z_1,\pm z_1',\dots,\pm z_m,\pm z_m'\}$,
				and if $n\equiv_41$,
				also include $\sigma(z_1')$.
	\end{enumerate}
	
\begin{center}
	\textbf{\textit Torus Shapes}
\end{center}
	Orthogonal slices of $\Psi(\Tsec_n)$
		along the $k$-axis
		are flat tori.
	These are all hexagons,
		with the following exceptions.
	For $i<m$:
		the sides contributed by $\pm z_i$
		and $\pm z_i'$
		bound a parallelogram at
		$k=\sqrt{\dfrac{-z_iz_i'}{\sigma(z_iz_i')}}$.
	If $n\not\equiv_41$,
		this also holds for $i=m$,
		at $k=\fu^4$
		(otherwise,
		the parallelogram bounded by the surfaces defined by $\pm z_m$
		and $\pm z_m'$
		crosses the $k$-axis
		above $\Psi(\Tsec_n)$).

\begin{samepage}
\begin{center}
	\textbf{\textit Gluing Maps}
\end{center}
	For each $\pm z_i$
		contributing sides as above,
		that pair of sides are attached toroidally by $\tau_{z_i}\in\U$.
	Finally,
		the tori in $\Tsec_n$
		at $k=1$
		and $k=\fu^4$
		are attached by the Anosov diffeomorphism $\eta_1=
		\begin{pmatrix}
			\fu & 0\\
			0 & \fu^{-1}
		\end{pmatrix}\in\D$.
\qed
\end{samepage}
\end{thm}

\begin{rem}\label{rem:sign}
	There will be repetition in the list $L$,
		as $z_{i+1}'$
		is always equal to one of $z_i$
		or $z_i'$.
	This notation agrees with that of Proposition \ref{prop:algorithm},
		where we tracked pairs of isometries by where the parallelogram
		tori occur.
\end{rem}

\begin{cor}
	Every Sol $3$-manifold
		is commensurable to a manifold admitting the fundamental domain
		described in Theorem \ref{thm:FDD}.
\end{cor}

\section{Examples and Visualization}
\label{sec:ex}

We conclude by showing some uses of Theorem \ref{thm:FDD}
	to study specific examples,
	focusing on the first two $n$-values modulo $4$:
	$n=2,3$ ($\not\equiv_41$),
	and $n=5,13$ ($\equiv_41$).
Table \ref{tab:fu}
	summarizes three pieces of information that will be important
	in the calculations.
First is the fundamental unit $\fu$
	of each field $\QQ(\sqrt n)$,
	which can be found in (or computed according to) 
	\cite{Azuhata1984}.
Second is the sign of the \emph{field norm}
	$$N:\ZZ_K\rightarrow\ZZ,\quad z\mapsto z\cdot\sigma(z)$$
	of each $\fu$,
	noting in particular that this decides the sign in
	$\sigma(\fu)=\pm\fu^{-1}$.
We have one of each possible signs in each of our cases for $n$
	modulo $4$.
Third is an expression showing how to write $\fu^2$
	as a linear combination of $1$
	and $\fu$
	over $\ZZ$,
	allowing us to easily reduce the degree
	of any product of such expressions.
\begin{table}[!htb]
	\caption{Some notes to compute our examples}\label{tab:fu}
	\begin{tabular}{|c|c|c|c|}
		\hline
		\boldmath$n$
			& \boldmath$\fu$
			& \boldmath$N(\fu)$
			& \boldmath$\fu^2$\\
		\hline
		$2$
			& $1+\sqrt2$
			& $-1$
			& $1+2\fu$\\
		\hline
		$3$
			& $2+\sqrt3$
			& $1$
			& $-1+4\fu$\\
		\hline
		$5$
			& $\frac{1+\sqrt5}{2}$
			& $1$
			& $1+\fu$\\
		\hline
		$13$
			& $\frac{3+\sqrt{13}}{2}$
			& $-1$
			& $1+3\fu$\\
		\hline
	\end{tabular}
\end{table}
	
Table \ref{tab:levels}
	shows the elements of $\ZZ_K$
	that contribute sides to $\Tsec_n$
	for $n=2$,
	$3$,
	$5$,
	and $13$.
The leftmost column indicates which example is being studied
	by reminding the reader of the corresponding fundamental unit $\fu$
	(where $\sqrt n$
	will appear).
The ``level'' column
	shows the locations going up the $k$-axis
	where orthogonal slices are parallelograms and,
	for $n\equiv_41$,
	also includes the hexagonal tops and bottoms
	(otherwise already included).
Equivalently,
	these are the levels along the $k$-axis
	where the edges of the $\Tsec_n$,
	formed by the intersecting sides,
	bifurcate as a $3$-valent graph.
The ``sides'' column lists the elements of $\ZZ_K$,
	up to sign,
	that bound the tori at the corresponding levels.
The rightmost column of shows
	how those sides correspond to the $z_i,z_i'$
	notation of Theorem \ref{thm:FDD}
	(and Proposition \ref{prop:algorithm}),
	where the indices range over the sequence of parallelograms.

Notice the suggestive occurrence of an additional square root
	(not in $K$)
	in the levels when $N(\fu)=-1$.
For the $n=5$
	example,
	we computed a way of expressing these as square roots of $\fu$.
For the $n=13$
	example,
	we see $\sqrt3$
	occurring in the denominators.

\begin{table}[!htb]
      \caption{levels of $\Tsec_n$
      	where edges bifurcate,
      	for $n=2$,
	$3$,
	$5$,
	and $13$}
	\label{tab:levels}
      \centering
	\bgroup
	\def\arraystretch{1.1}
	\begin{tabular}{|c|c|c|c|}
		\hline
		\boldmath$\fu$ & {\bf level} & {\bf sides} & \boldmath$i$\\
			\hline
		\multirow{5}{*}{$1+\sqrt 2$}
			& 1 & $1,\ \sqrt2$ & $1$\\
			\cline{2-4}
			 & $\fu$ & 1,\ $\fu$ & $2$\\
			\cline{2-4}
			 & $\fu^2$ & $\fu\sqrt2,\ \fu$ & $3$\\
			\cline{2-4}
			  & $\fu^3$ & $\fu^2,\ \fu$ & $4$\\
			\cline{2-4}
			 & $\fu^4$ & $\fu^2,\ \fu^2\sqrt2$ & $5=m$\\
			\hline
		\multirow{7}{*}{$2+\sqrt3$}
			& 1 & $1,\ \sqrt3$ & $1$\\
			\cline{2-4}
			& $\fu^{1/2}$ & $1,\ 1+\sqrt3$ & $2$\\
			\cline{2-4}
			& $\fu^{3/2}$ & $\fu,\ 1+\sqrt3$ & $3$\\
			\cline{2-4}
			& $\fu^2$ & $\fu,\ \fu\sqrt3$ & $4$\\
			\cline{2-4}
			& $\fu^{5/2}$ & $\fu,\ \fu(1+\sqrt3)$ & $5$\\
			\cline{2-4}
			& $\fu^{7/2}$ & $\fu^2,\ \fu(1+\sqrt3)$ & $6$\\
			\cline{2-4}
			& $\fu^4$ & $\fu^2,\ \fu^2\sqrt3$ & $7=m$\\
			\hline
		\multirow{4}{*}{$\dfrac{1+\sqrt5}{2}$}
			& $1$ & $1,\ \fu,\ \fu^{-1}$ & (hexagonal)\\
			\cline{2-4}
			 & $\fu$ & $1,\ \fu$ & $1$\\
			\cline{2-4}
			& $\fu^3$ & $\fu^2,\ \fu$ & $2=m-1$\\
			\cline{2-4}
			& $\fu^4$ & $\fu^3,\ \fu^2,\ \fu$ & (hexagonal)\\
			\hline
		\multirow{9}{*}{$\dfrac{3+\sqrt{13}}{2}$}
			& $1$ & $1,\ \frac{1+\sqrt{13}}{2},\ \frac{1-\sqrt{13}}{2}$
				& (hexagonal)\\
			\cline{2-4}
			& $\frac{1+\sqrt{13}}{2\sqrt3}$
				& $1,\ \frac{1+\sqrt{13}}{2}$ & $1$\\
			\cline{2-4}
			 & $\fu$ & $1,\ \fu$ & 2\\
			\cline{2-4}
			& $\frac{\fu(1+\fu)}{\sqrt3}$ & $1+\fu,\ \fu$ & $3$\\
			\cline{2-4}
			& $\frac{\fu(1+2\fu)}{\sqrt3}$ & $1+2\fu,\ \fu$ & $4$\\
			\cline{2-4}
			& $\fu^3$ & $\fu^2,\ \fu$ & $5$\\
			\cline{2-4}
			& $\frac{\fu^2(1+4\fu)}{\sqrt3}$ & $\fu^2,\ 1+4\fu$
				& $6=m-1$\\
			\cline{2-4}
			& $\fu^4$ & $\fu^2,\ 1+4\fu,\ 2+7\fu$ & (hexagonal)\\
			\hline
	\end{tabular}
	\egroup
\end{table}

Figure \ref{fig:towers}
	shows plots of the cusp sections $\Tsec_n$
	for these same examples under the map $\Psi$,
	but with the third coordinate scaled logarithmically in the base $\fu$.
That is,
	we are seeing the images of the $\Tsec_n$
	under the map 
\begin{align*}
	\Psi_n:&\ \HtH\rightarrow\RR^3,\quad
		(x_1+y_1i,x_2+y_2i)\mapsto
		\left(x_1,x_2,\log_\fu\big(y_1/y_2\big)\right).
\end{align*}
Creating such plots only requires generating a list of the sides,
	which is a less arduous computation than finding the levels
	of the parallelograms.
Those parallelograms will,
	of course,
	occur on their own as a result of how the sides intersect.
\begin{figure}[!htb]
	\centering
	\includegraphics[width=\textwidth]{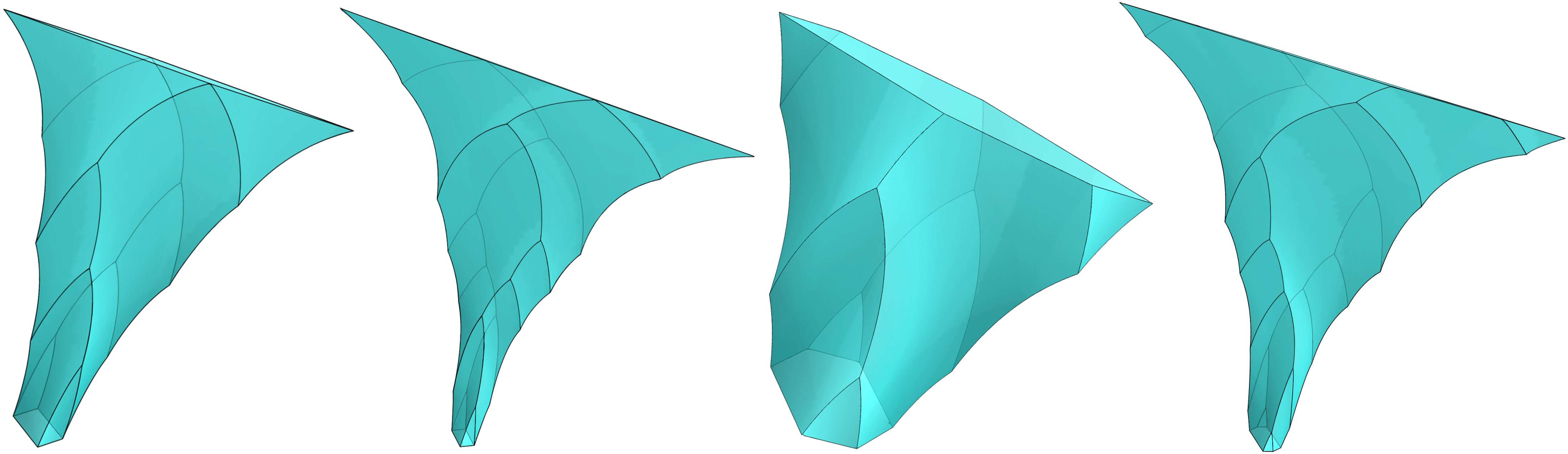}
		\caption{\cite{Mathematica}
			Computer generated images are shown of
				the images of $\Psi_n(\Tsec_n)$,
				for $n=2, 3, 5$ and $13$,
				from left to right,
				respectively.}
	\label{fig:towers}
\end{figure}

The images of $\Psi_n(\Tsec_n)$
	become arbitrarily more complicated and stretched
	(though not uniformly) as $n$
	increases,
	making rendering the graphics less feasible for many other cases.
Another way to study patterns in these shapes,
	for which computer images are not necessary,
	is to look at which pairs of sides persist
	from one parallelogram to the next,
	going up the $k$-axis,
	and that can be examined via the relatively easy computation
	of the sides.
One interesting approach is to first generate a list
	like those under the ``sides'' columns of Table \ref{tab:levels},
	where persisting sides maintain their position among the entries
	they appear in,
	and then to draw edges between the different vertical stacks of
	the same integers.
This creates a $3$-valent
	graph that is equivalent (just flipped upside down)
	to the one seen on the front of the plot
	of each $\Psi_n(\Tsec_n)$
	in Figure \ref{fig:towers},
	which repeats symmetrically on the back.
	
This thereby offers some new and fairly accessible
	techniques by which one might further study quadratic fields,
	cusp sections of Hilbert-Blumenthal surfaces,
	and
	commensurability classes of Sol $3$-manifolds.

\section{Acknowledgements}
	This research paper has been made possible
		thanks to the financial support generously given by the
		FORDECyT-CONACyT (Mexico) grant \#265667,
		Universidad Nacional Aut\'onoma de M\'exico.
	The second author was financed by grant IN106817,
		PAPIIT, DGAPA, Universidad Nacional Aut\'onoma de M\'exico.
	The authors also express their gratitude to
		Ian Agol,
		Kathleen Byrne,
		Jesse Ira Deutsch,
		Paul Garrett,
		Ben McReynolds,
		Jorge Millan
		and Walter Neumann
		for helpful suggestions and discussion;
		to the reviewer for their detailed comments and corrections;
		and to Dennis Ryan and Simon Woods
		for help with creating the computer generated images.
\bibliographystyle{plain}
\bibliography{references}	
\end{document}